\documentclass[11pt, oneside]{amsart}   	
\usepackage{geometry}                		
\usepackage{graphicx}				
\usepackage{amssymb}
\usepackage{hyperref}
\usepackage{tikz}

\newtheorem{theorem}{Theorem}[section]
\newtheorem{lemma}[theorem]{Lemma}
\newtheorem{cor}[theorem]{Corollary}

\theoremstyle{definition}
\newtheorem{definition}[theorem]{Definition}
\newtheorem{example}[theorem]{Example}

\theoremstyle{remark}

\numberwithin{equation}{section}

\title{Sharp Bounds for Neighborhood degree based indices of Graphs}
\author{Sanju Vaidya}
\address{Math \& CS Department, Mercy University, 555 Broadway, Dobbs Ferry, NY 10522}
\curraddr{}
\email{SVaidya@mercy.edu}
\thanks{}

\author{Jeff Chang}
\address{Math \& CS Department, Mercy University, 555 Broadway, Dobbs Ferry, NY 10522}
\curraddr{}
\email{cchang4@mercy.edu}
\thanks{}


\begin{document}
\maketitle

\begin{abstract}

In this paper, we will construct formulas and bounds for Neighborhood Degree-based indices of graphs and describe graphs that  attain the bounds. Furthermore, we will establish a lower bound for the spectral radius of any graph.  

\end{abstract}

\section{Introduction}

 In the last forty years, many scientists have developed mathematical models for analyzing structures and properties of various chemical compounds. Graph theory plays a very important role in developing many types of models such as Quantitative Structure-Property Relationships (QSPR) models, and Quantitative Structure-Activity Relationships (QSAR) models (\cite{BGG},\cite{GFS}, \cite{DB},\cite{TC}). In molecular graphs of chemical compounds, vertices correspond to atoms and edges correspond to the bonds between them. A topological index (connectivity index) is a type of a molecular descriptor that is based on the molecular graph of a chemical compound. In 1947 Harry Weiner introduced a topological index related to molecular branching. He correlated the indices with the boiling points of certain chemical compounds, alkanes. This inspired many mathematicians and chemists to develop more topological indices for molecular graphs.

In 1972, Zagreb indices were introduced by Gutman and Trijanastic \cite{GT}. They are based on the degrees of the vertices. They are very useful in modeling chemical and biological properties of chemical compounds, as shown by Klein, D. J. \cite{DB}, Basak et al \cite{BGG}, and Todeschin et al \cite{TC}. Gutman et al \cite{GFMG} and Zhou et al \cite{B} used Zagreb indices for analyzing bounds for Estrada index, which is based on the eigenvalues of the adjacency matrix.  

Zagreb indices are generalized by many mathematicians and scientists. In 2019, Mondal et al \cite{MP} introduced the general neighborhood Zagreb index and proved its importance in predicting properties of chemical compounds. This index is based on the neighborhood degree of a vertex, which is the sum of the degrees of its adjacent vertices. In the last five years, many scientists used the neighborhood degree based indices to model various properties of chemical compounds which are used in medicines.  For example, Mondal et al  \cite{MP2}, investigated neighborhood degree based indices of some chemical compounds like remdesivir (GS-5734) which is useful for treatment of Covid-19 patients. In 2023, Balasubramaniyan and Chidambaram \cite{BC}, used these indices for analysis of asthma drugs. 

The main research question is how to compute bounds and formulas of the indices stated above and find graphs which attain the bounds. In this paper, we will establish formulas for the general neighborhood Zagreb index. Additionally, we will find upper and lower bounds, and describe graphs which attain the bounds.    Hu-Li-Shi-Xu-Gutman \cite{HLSXG} investigated the zeroth-order general Randic index (which is the same as general Zagreb index) for molecular graphs in which the maximum vertex degree is at most 4. In Theorem \ref{T:NB2} of  this paper we will extend this further to the general  neighborhood Zagreb index for graphs with maximum vertex degree more than 4.

In 2004,  Yu-Lu-Tian \cite{YLT} established a lower bound for the spectral radius of any graph, which is the largest eigenvalue of its adjacency matrix. They used the 2-degree of a vertex, which is the sum of the degrees of the adjacent vertices. So the 2-degree is the same as the neighborhood degree of a vertex. This is amazing! Using this and our lower bound for general neighborhood Zagreb index, we will find a lower bound for spectral radius of any graph in Section 4. Additionally, using generalization of Benakli, Halleck, and Kingan \cite{BHK}, we could define 2-distance neighborhood degree, $\delta_2(v)$, of any vertex v as the sum of the degrees of the vertices which are at distance 2 from the vertex v, which opens more research problems. 

In Section 2 of this paper, we will review definitions of the first Zagreb index and general neighborhood Zagreb index.   In Section 3 we will establish formulas of the general neighborhood Zagreb index for graphs. Additionally, we will find upper and lower bounds for the index and determine graphs which attain the bounds. In addition, we will establish sharp upper and lower bounds for the index for some special graphs. In Section 5 we will have discussion and conclusion.

\section{Review of Zagreb indices}

We will use the notation and terminology introduced in Gutman et al \cite{GFMG} and Mondal et al \cite{MP}.

 In {\bf molecular graphs} of chemical compounds, vertices correspond to individual atoms and edges correspond to the bonds between them. A {\bf topological index} (connectivity index) is a type of a molecular descriptor that is based on the molecular graph of a chemical compound.The definition of the first Zagreb index, which is introduced by Gutman and Trijanastic \cite{GT}, is as follows.
\begin{definition}
    Let $G$ be a graph with vertex set $V(G)$ and $d(u)$ be the degree of the vertex $u\in V(G)$. Then the first Zagreb index is defined as follows:
    \[Z_G=M_1(G)=\sum_{u\in V(G)}d(u)^2.\]
\end{definition}

In 2007 Gutman-Fortula-Marković-Glišić \cite{GFMG} proved that for a chemical tree (alkane) with $n$ vertices $Z_G=M_1(G)=6n-10-2n_2-2n_3$, where $n_i$ is the number of vertices of degree $i$ for i = 2, 3. In Vaidya and Surendran \cite{VS}, we established formulas of the first Zagreb index for molecular graphs of cycloalkanes, alkenes, and alkynes. 

Zagreb indices have been generalized and studied for more than 30 years. In 2019, Mondal et al \cite{MP} introduced the following general neighborhood Zagreb index.

\begin{definition}
    Let $G$ be a graph with vertex set $V(G)$. For a vertex $u\in V(G)$, let  $\delta_G(u)$ be its neighborhood degree, which is the sum of the degrees of the vertices which are adjacent to the vertex $u$. Let $\alpha$ be a nonzero real number such that $\alpha\neq 1$. Then the general neighborhood Zagreb index is defined as follows:
    \[NM_\alpha (G)=\sum_{u\in V(G)}\delta_G(u)^\alpha.\]
\end{definition}

\section{Formulas and bounds for Neighborhood Degree based indices}

In this section we will establish some formulas for the general neighborhood Zagreb index. We will then find sharp upper and lower bounds for it using these formulas. First, we need the following Lemma which is proved in Vaidya-Chang \cite{VC}.

\begin{lemma}\label{Lemma}
Let $p$ and $q$ be any positive integers such that $p < q$. Let $\alpha\in\mathbb{R}$ and $\alpha\not\in\{0,1\}$. Then we have the following.
\begin{enumerate}
\item If  the real number $\alpha<0$ or $\alpha>1$, then  $(p + i)^\alpha - p^\alpha - i(\frac{q^\alpha - p^\alpha}{q - p}) \leq 0$ for $1 \leq i \leq q - p - 1$  and if $0 < \alpha < 1$, then $(p + i)^\alpha - p^\alpha - i(\frac{q^\alpha - p^\alpha}{q - p}) \geq 0$ for $1 \leq i \leq q - p - 1$. 

\item if  the real number $\alpha<0$ or $\alpha>1$, then $(p + i)^\alpha - p^\alpha  - i((p + 1)^\alpha - p^\alpha) \geq 0$ for $2 \leq i \leq q - p$ and for $0 < \alpha < 1$, $(p + i)^\alpha - p^\alpha  - i((p + 1)^\alpha - p^\alpha) \leq 0$ for $2 \leq i \leq q - p$.  

\end{enumerate}
\end{lemma}

\begin{theorem}\label{T:NB}
    Let $G$ be a graph with $n$ vertices and $m$ edges where $n\geq 3$. Let $n_i$ denote the number of vertices of neighborhood degree $i$. Let $\delta$ and  $\Delta$ respectively denote the minimum and the maximum neighborhood degree.  Assume that $\delta \not= \Delta$. Let $\alpha\in\mathbb{R}$ and $\alpha\not\in\{0,1\}$.  Further denote $s_\alpha=\frac{\Delta^\alpha- \delta^\alpha}{\Delta- \delta}$. Then we have the following.
    \begin{enumerate}
        \item The general neighborhood Zagreb index
        \[NM_{\alpha}(G)=n\delta^\alpha +  (M_1 - n\delta)s_\alpha +\sum_{i =1}^{\Delta - \delta - 1}n_{\delta + i} \left[(\delta + i)^\alpha - \delta^\alpha  - is_\alpha\right].\]

Additionally, if  $\alpha<0$ or $\alpha>1$ the coefficients  $(\delta + i)^\alpha - \delta^\alpha - is_\alpha \leq 0$ for $1 \leq i \leq \Delta -\delta - 1$  and for $0 < \alpha < 1$, the coefficients $(\delta + i)^\alpha - \delta^\alpha - is_\alpha \geq 0$ for $1 \leq i \leq \Delta -\delta - 1$.

        \item The general neighborhood Zagreb index

        \begin{align*}
            NM_{\alpha}(G)=n\delta^\alpha +  (M_1 - n\delta) (&(\delta + 1)^\alpha - \delta^\alpha)  +\\
            & \sum_{i = 2}^{\Delta - \delta }n_{\delta + i} \left[(\delta + i)^\alpha - \delta^\alpha  - i((\delta + 1)^\alpha - \delta^\alpha)\right].
        \end{align*}
 
Additionally, if  $\alpha<0$ or $\alpha>1$, the coefficients $(\delta + i)^\alpha - \delta^\alpha  - i((\delta + 1)^\alpha - \delta^\alpha) \geq 0$ for $2 \leq i \leq \Delta -\delta$ and for $0 < \alpha < 1$, the coefficients $(\delta + i)^\alpha - \delta^\alpha  - i((\delta + 1)^\alpha - \delta^\alpha) \leq 0$ for $2 \leq i \leq \Delta -\delta $.      
    \end{enumerate}
\end{theorem}
\begin{proof}
We have $\sum_{i= \delta}^\Delta in_i= M_1$ by part (i) of Lemma 1 of  \cite{MDP}. We also have $n=\sum_{i= \delta}^\Delta n_i$. Solving these equations, we will rewrite $n_\Delta$ as follows.
\[
n_\Delta=\frac{M_1 - n\delta - \sum_{i=\delta + 1}^{\Delta  - 1}((i - \delta)n_i}{\Delta - \delta}.
\]

Write $NM_{\alpha} (G)$ in terms of $n_i$, $\delta\leq i\leq \Delta$, and substitute $n_\delta$ with $n-\sum_{i=\delta + 1}^\Delta n_i$, we obtain
\[
NM_{\alpha} (G)=\sum_{i=\delta}^\Delta  i^\alpha n_i = \delta^\alpha( n - \sum_{i= \delta + 1}^\Delta n_i) + \sum_{i=\delta + 1}^\Delta  i^\alpha n_i = n\delta^\alpha + \sum_{i = \delta + 1}^\Delta  n_i(i^\alpha - \delta^\alpha).
\]
Further substitute $n_\Delta$ into the sum and using Lemma \ref{Lemma}, we get the result in part (1). Using the identities $n_\delta+n_{\delta + 1} = n - \sum_{i = \delta + 2}^\Delta n_i$ and $\delta n_\delta + (\delta + 1)n_{\delta + 1} = M_1 - \sum_{i = \delta + 2}^\Delta  in_i$, we solve for $n_\delta$ and $n_{\delta + 1}$. Substituting $n_\delta$ and $n_{\delta + 1}$ into $NM_{\alpha}(G)=\sum_{i=\delta}^\Delta  i^\alpha n_i$, and using Lemma \ref{Lemma} again, we get the result in part (2). 

\end{proof}

The following corollary gives upper bounds or lower bounds for the general Zagreb index, depending on the values of $\alpha$. Both the bounds are sharp. 

\begin{cor}\label{C:NB}
    Let $G$ be a graph with $n$ vertices and $m$ edges where $n\geq 3$. Let $n_i$ denote the number of vertices of neighborhood degree $i$. Let $\delta$ and  $\Delta$ respectively denote the minimum and the maximum neighborhood degree.  Assume that $\delta \not= \Delta$. Let $\alpha\in\mathbb{R}$ and $\alpha\not\in\{0,1\}$.  Further denote $s_\alpha=\frac{\Delta^\alpha- \delta^\alpha}{\Delta- \delta}$. Then we have the following. 

\begin{enumerate}
        \item If $\alpha<0$ or $\alpha>1$, the upper bound for the general neighborhood Zagreb index is $ NM_{\alpha}(G) \leq n\delta^\alpha +  (M_1 - n\delta)s_\alpha$.

    \item If $0<\alpha<1$, it becomes the lower bound for the general first Zagreb index, that is, $ NM_{\alpha}(G) \geq n\delta^\alpha +  (M_1 - n\delta)s_\alpha$.
    
Moreover, for parts (1) and (2), the equality occurs for any graph with vertices of neighborhood degrees $\delta$ and $\Delta.$

\item If $\alpha<0$ or $\alpha>1$, then the lower bound for the general neighborhood Zagreb index is $NM_{\alpha}(G) \geq n\delta^\alpha +  (M_1 - n\delta) ((\delta + 1)^\alpha - \delta^\alpha)  + n_{\Delta} (\Delta^\alpha - \delta^\alpha - (\Delta - \delta) ((\delta + 1)^\alpha - \delta^\alpha))$.

\item If $0<\alpha<1$, then the upper bound for the general neighborhood Zagreb index is $NM_{\alpha}(G) \leq n\delta^\alpha +  (M_1 - n\delta) ((\delta + 1)^\alpha - \delta^\alpha)  + n_{\Delta} (\Delta^\alpha - \delta^\alpha - (\Delta - \delta) ((\delta + 1)^\alpha - \delta^\alpha))$.
   
Moreover, for parts (3) and (4), the equality occurs if the graph G is a path.
    
\end{enumerate}
       
\end{cor}
\begin{proof}
Follows from Theorem \ref{T:NB}. 
\end{proof}

\begin{example}
Here is an example. In Figure \ref{fig:1}, $\delta=4, \Delta = 10, m=n=12$ and all vertices have either neighborhood degree $4$ or $10$. It attains the bounds given in parts (1) and (2) of  Corollary \ref{C:NB}.
\begin{figure}[h]
\centering
\begin{minipage}{.4\textwidth}
  \centering
  \begin{tikzpicture}[scale=0.3]
\draw[ultra thick] (-3,3) -- (3,3) -- (3,-3)--(-3,-3) -- (-3,3);
\draw[ultra thick] (3,3)--(6,4);
\draw[ultra thick] (3,3)--(4,6);
\draw[ultra thick] (-3,3)--(-4,6);
\draw[ultra thick] (-3,3)--(-6,4);
\draw[ultra thick] (3,-3)--(6,-4);
\draw[ultra thick] (3,-3)--(4,-6);
\draw[ultra thick] (-3,-3)--(-6,-4);
\draw[ultra thick] (-3,-3)--(-4,-6);
\filldraw[black] (3,3) circle (10pt);
\filldraw[black] (-3,3) circle (10pt);
\filldraw[black] (3,-3) circle (10pt);
\filldraw[black] (-3,-3) circle (10pt);
\filldraw[black] (6,4) circle (10pt);
\filldraw[black] (4,6) circle (10pt);
\filldraw[black] (-6,4) circle (10pt);
\filldraw[black] (-4,6) circle (10pt);
\filldraw[black] (6,-4) circle (10pt);
\filldraw[black] (4,-6) circle (10pt);
\filldraw[black] (-6,-4) circle (10pt);
\filldraw[black] (-4,-6) circle (10pt);
\end{tikzpicture}
\caption{}
  \label{fig:1}
\end{minipage}%

\end{figure}
\end{example}

In Theorem (2.2) of \cite{HLSXG}, Hu-Li-Shi-Xu-Gutman  established formulas for the maximum and the minimum zeroth-order general Randic index (which is same as general Zagreb index) for $(n, m)$ molecular graphs if $2m - n$ is congruent to $0, 1,$ or $2$ modulo $3$ and the maximum vertex degree is $4$. In the following Theorem, we will extend this and find sharp bounds for the general neighborhood Zagreb index if $M_1 - n \delta$ is congruent to any integer $r$ modulo $\Delta - \delta$ where $0 \leq r \leq \Delta - \delta- 1$ and the maximum neighborhood degree $\Delta \geq 3 $.

\begin{theorem}\label{T:NB2}
    Let $G$ be a graph with $n$ vertices and $m$ edges where $n\geq 3$. Let $n_i$ denote the number of vertices of neighborhood degree $i$. Let $\delta$ and  $\Delta$ respectively denote the minimum and the maximum neighborhood degree. Assume that $\Delta - \delta\geq 2$ and $M_1 - n\delta = q(\Delta - \delta) + r$, where q is a positive integer and r is an integer such that $0 \leq r \leq \Delta - \delta - 1$. Let $\alpha\in\mathbb{R}$ and $\alpha\not\in\{0,1\}$.  Further denote $s_\alpha=\frac{\Delta^\alpha- \delta^\alpha}{\Delta-\delta}$.  Then we have the following.

 \begin{enumerate}
        \item If $r = 0$ and $n_{\Delta} = q$, then the graph G is a bi-neighborhood degree graph with vertices of neighborhood degree $\delta$ and $\Delta.$. 

  \item If $r \geq 1$ and $n_{\Delta} = q$, then $n_i = 0$ for $\delta + r + 1 \leq i \leq \Delta  - 1$ and $n_{\delta + r} \leq 1$.

\item If $r \geq 1$ and $\alpha<0$ or $\alpha>1$ and $n_{\delta + r } \not = 0$ , then the upper bound for the general neighborhood Zagreb index is
 $NM_{\alpha}(G)\leq n\delta^\alpha +  (M_1 - n\delta)s_\alpha +\ (\delta + r)^\alpha - \delta^\alpha  - rs_\alpha$.

   \item  If $r \geq 1$ and $0<\alpha<1$ and $n_{\delta + r} \not = 0$, it becomes the lower bound for the general neighborhood Zagreb index, that is,  $NM_{\alpha}(G)\geq n\delta^\alpha +  (M_1 - n\delta)s_\alpha +\ (\delta + r)^\alpha - \delta^\alpha  - rs_\alpha$.

 \end{enumerate}
        
Additionally, in both parts (3) and (4) the equality occurs if the graph G has $q$ vertices of neighborhood degree $\Delta$, a unique vertex of neighborhood degree $\delta + r$, and $n - q - 1$ vertices of neighborhood degree $\delta$.
 
\end{theorem}

\begin{proof}
We have $\sum_{i= \delta}^\Delta in_i= M_1$ by part (i) of Lemma 1 of  \cite{MDP}. We also have $n=\sum_{i= \delta}^\Delta n_i$. By solving the equations, we get $\sum_{i=\delta + 1}^\Delta (i - \delta)n_i= M_1 - n\delta = q(\Delta - \delta) + r$. If $r = 0$ and $n_{\Delta} = q$, then $n_i = 0$ for $ \delta + 1 \leq i \leq \Delta - 1$. So we get (1). Also if $r \geq 1$ and $n_{\Delta} = q$,then $n_i = 0$ for $\delta + r + 1 \leq i \leq \Delta - 1$ and $n_{\delta + r} \leq 1$. So we get (2). To prove (3), let $r\geq 1$ and $n_{\delta + r } \not =  0$.
By part (1) of Theorem \ref{T:NB} we have the general neighborhood Zagreb index 
$NM_{\alpha}(G)=n\delta^\alpha +  (M_1 - n\delta)s_\alpha +\sum_{i =1}^{\Delta - \delta - 1}n_{\delta + i} \left[(\delta + i)^\alpha - \delta^\alpha  - is_\alpha\right].$ If $\alpha<0$ or $\alpha>1$, then $(\delta + i)^\alpha - \delta^\alpha - is_\alpha \leq 0$ for $1 \leq i \leq \Delta -\delta - 1$. So $NM_{\alpha}(G) \leq n\delta^\alpha +  (M_1 - n\delta)s_\alpha + n_{\delta + r} [(\delta + r)^\alpha - \delta^\alpha  - (r)s_\alpha].$

 Since $n_{\delta + r } \not =  0$, the result in part (3) follows. Similarly, the result in part (4) follows since for $0<\alpha<1$, $(\delta + i)^\alpha - \delta^\alpha - is_\alpha \geq 0$ for $1 \leq i \leq \Delta -\delta - 1$. In addition, by part (1) of Theorem (3.2), for both parts (3) and (4) the equality occurs if the graph G has $q$ vertices of neighborhood degree $\Delta$, a unique vertex of neighborhood degree $\delta + r$, and $n - q - 1$ vertices of neighborhood degree $\delta$.

\end{proof}

\begin{example}
For the graph in Figure \ref{fig:2},  we have $n = 5, \delta = 3, \Delta = 5$ and $M_1 - n\delta = 7 = 3(\Delta - \delta) + 1$. The graph has $q =3$ vertices of neighborhood degree $\Delta = 5$, a unique vertex of neighborhood degree $\delta + r = 4$, and the remaining $n - q - 1$ vertices of neighborhood degree $\delta = 3$. It attains the bounds in parts (3) and (4) of Theorem \ref{T:NB2}.

\begin{figure}
\begin{minipage}{.4\textwidth}
  \centering
  \begin{tikzpicture}[scale=0.3]
\draw[ultra thick] (-3,3) -- (3,3) -- (3,-3)--(-3,-3) -- (-3,3);
\draw[ultra thick] (3,3)--(6,4);
\filldraw[black] (3,3) circle (10pt);
\filldraw[black] (-3,3) circle (10pt);
\filldraw[black] (3,-3) circle (10pt);
\filldraw[black] (-3,-3) circle (10pt);
\filldraw[black] (6,4) circle (10pt);
\end{tikzpicture}
\caption{}
  \label{fig:2}
\end{minipage}
\end{figure}
\end{example}

\section{Spectral Radius, neighborhood degree, and its generalization}

In 2004, Yu-Lu-Tian \cite{YLT} established a lower bound for the spectral radius, $\rho (G)$  of any graph $G$, which is the largest eigenvalue of its adjacency matrix. They used the 2-degree of a vertex, which is the sum of the degrees of the adjacent vertices. So the 2-degree is the same as the neighborhood degree of a vertex. They proved that the $\rho^2 \geq  \frac {NM_2(G)}{M_1}$. Using this and part (2) of Theorem \ref{T:NB}, we get the following Theorem.

\begin{theorem}
Let $G$ be a graph with $n$ vertices and $m$ edges where $n\geq 3$. Let $\delta$ denote the minimum neighborhood degree. Then the spectral radius,  
$\rho^2 \geq  \frac {M_1(2\delta +1) - n\delta^2 - n\delta}{M_1}$.

\end{theorem}

\begin{proof}
Follows from Theorem 4 of Yu-Lu-Tian \cite{YLT} and part 2 of Theorem \ref{T:NB}.
\end {proof}

By using generalization of \cite{BHK}, we could define 2-distance neighborhood degree, $\delta_2(v)$, of any vertex $v$ as the sum of the degrees of the vertices which are at distance 2 from the vertex $v$. Further we could define the general 2-distance neighborhood degree index as $NM2_\alpha (G)=\sum_{u\in V(G)}\delta_2(u)^\alpha$. In the following Theorem, we will establish some formulas for the 2-distance neighborhood degree index if the diameter of the graph is 2. 

\begin{theorem}
    Let $G$ be a graph of diameter 2 with $n$ vertices and $m$ edges where $n\geq 3$. Let $n_i$ denote the number of vertices of 2-distance neighborhood degree $i$. Let $d$ and  $D$ respectively denote the minimum and the maximum 2-distance neighborhood degree.  Assume that $ d \not = 0$ and $d \not= D$. Let $\alpha\in\mathbb{R}$ and $\alpha\not\in\{0,1\}$.  Further denote $S_\alpha=\frac{D^\alpha- d^\alpha}{D - d}$. Then we have the following.
    \begin{enumerate}
        \item The 2-distance neighborhood degree index
        \[NM2_\alpha(G)=nd^\alpha +  (2m(n - 1) - M_1 - nd)S_\alpha +\sum_{i =1}^{D - d - 1}n_{d + i} \left[(d + i)^\alpha - d^\alpha  - iS_\alpha\right].\]

        \item The 2-distance neighborhood degree index

        \begin{align*}
            NM2_\alpha(G)=nd^\alpha +  (2m(n - 1) - M_1 &- nd) ((d + 1)^\alpha - d^\alpha)+  \\
            &\sum_{i = 2}^{D - d }n_{d + i} \left[(d + i)^\alpha - d^\alpha  - i((d + 1)^\alpha - d^\alpha)\right].
        \end{align*}

    \end{enumerate}
\end{theorem}
\begin{proof}
We have $n=\sum_{i= d}^D n_i$ and $\sum_{i= d}^D in_i= 2m(n - 1) - M_1$. Solving these equations for $n_d$ and $n_D$ we get (1). Additionally, by solving the equations for  $n_d$ and  $n_{d + 1}$, we get (2).

\end{proof}

\section{Discussion and Conclusion}

In conclusion, we established formulas of the general neighborhood Zagreb index for graphs. Additionally, we found upper and lower bounds for the general neighborhood Zagreb index and determine graphs which attain the bounds. In addition, we established sharp upper and lower bounds for the general  neighborhood Zagreb index for some special graphs. We also found lower bound for the spectral radius. Additionally, we discussed generalization of the neighborhood degree which leads to more research problems.

Since there are more than 100 topological indices, there are many open problems of finding bounds and formulas for them. They play a crucial role in modeling various properties of chemical compounds. This process of mathematical modeling is very useful in real life because if reduces cost and saves time in developing medicines. It is fascinating to see that neighborhood degree based indices are useful in finding bounds for spectral radius as well as predicting physico-chemical properties of chemical compounds.

\bibliographystyle{amsplain}

\end{document}